\documentclass[12pt]{elsarticle}
\usepackage{float}
\usepackage[utf8]{inputenc}
\usepackage{upgreek}
\usepackage{amssymb}
\usepackage{graphicx}
\usepackage{epstopdf}
\usepackage{latexsym}
\usepackage{amsmath}
\usepackage{amssymb}
\usepackage{latexsym}
\usepackage{amsmath}
\usepackage{amsfonts}
\usepackage{amsfonts}
\usepackage{array}
\usepackage{microtype}
\usepackage{amsthm}
\usepackage{hyperref}
\newtheorem{definition}{\bf Definition}[section]
\newtheorem{lemma}[definition]{\bf Lemma}
\newtheorem{theorem}[definition]{\bf Theorem}

\newtheorem{corollary}[definition]{\bf Corollary}

\newtheorem{remark}[definition]{\bf Remark}
\newtheorem{example}[definition]{\bf Example}
\usepackage{hyperref}
\usepackage{xcolor}

\usepackage[symbol]{footmisc}
\def\correspondingauthor{\footnote{*Corresponding author.}}

\begin{document}

\begin{frontmatter}


\title{Korovkin-type approximation results through summability methods in the Applied Science}



\author{M. del Carmen List\'an-Garc\'{\i}a}
\address{Department of Mathematics
University of C\'adiz, CASEM, Pol. R\'{\i}o San Pedro s/n, 11510-Puerto Real. Spain.} \ead{mariadelcarmen.listan@uca.es}

\author{Mar\'{\i}a Pilar Romero de la Rosa\correspondingauthor{*}}

\address{ Department of Mathematics
University of C\'adiz Avda. de la Universidad s/n 11403-Jerez de
la Frontera. Spain.}
\ead{pilar.romero@uca.es}

\begin{abstract}
We present  Korovkin approximation theorems that incorporate summability methods. These result allows us to obtain a unified treatment of  several previous results, focusing on the underlying structure and the properties that a summability method should satisfy in order to establish a Korovkin-type approximation result. As a by-product we obtain new Korovkin-type results incorporating summability methods, for example for  regular matrix convergence methods, convergence through ideals of natural numbers, etc, and we provide further directions and a preparation for questions which remain open.
\end{abstract}


\end{frontmatter}

\section{Introduction}

Throughout this paper, $\mathbb{N}$ will denote the set of natural numbers. Let $X$ be a  Banach space.
A linear summability method in $X$ (or convergence method) will be denoted by $\mathcal{R}$; that is, ${\mathcal R}$ will be a linear map ${\mathcal R} : \mathcal{D}_{\mathcal R} \subset X^{\mathbb{N}}\longrightarrow X $ (here ${\mathcal D}_{{\mathcal R}}$ denotes the domain of ${\mathcal R}$). Thus, a sequence $(x_n)\in X^{\mathbb{N}}$ is ${\mathcal R}$-convergent to $L$ (and it will be denoted by $x_n\overset{\mathcal{R}}{\longrightarrow} L$ ) provided ${\mathcal R}((x_n)_{n\geq 1})=L$.

With the development of Fourier theory, other convergence methods of the series were studied which are interesting in their own right. Convergence methods have generated so much interest in Approximation Theory and Applied Mathematics that different monographs have appeared in the literature \cite{mursaleenbook}; moreover, this is a very active field of research with many contributors.

The Korovkin approximation theorem is one of the best known and most useful results in Approximation Theory, so beautiful that it has attracted the interest of many mathematicians and it has been strengthened and generalized in many directions. Some recent monographs by  F. Altomare  \cite{altomare,altomare2} evidence  this fact.
 One of the classical versions \cite{altomare2} states that if $L_n$ is a sequence of positive linear operators on the space of continuous functions $\mathcal{C}([0,1])$  endowed with the supremum norm into $\mathcal{C}([0,1])$, then for all $f\in\mathcal{C}[0,1]$, $\|L_nf-f\|_{\infty}\to 0$ if and only if $\|L_n(e_i)-e_i\|_\infty\to 0$, for only three continuous functions $e_1(x)=1,e_2(x)=x$ and $e_3(x)=x^2$.

A sequence $(x_n)\subset X$  is said to be statistically convergent to $L$ if for any $\varepsilon$ the subset $\{n\,:\, \|x_n-L\|>\varepsilon\}$ has zero density on $\mathbb{N}$. The term statistical convergence was first presented by Fast \cite{Fast} and Steinhaus \cite{Stein} independently in the same year 1951.
 Gadjiev and Orhan  \cite{orhan} obtained a remarkable  version of Korovkin's approximation theorem in terms of the statistical convergence. Gadjiev and Orhan's result  really improve the result of Korovkin: they exhibited examples of positive operators that do not satisfy the classical theorem with uniform convergence, but that  satisfy the statistical version. Since then, Gradiev-Orhan's research line have been extended  for different types of convergence methods (\cite{korov, b3,EMN,MA, MVEG}).
 
 A good source of problems consists in considering a classical result  which is true for the usual convergence and, to try to prove it, replacing the usual convergence by other convergence methods. 
In this way, it is possible to see a classical result from a new point of view. Sometimes, it is possible to characterize those summability methods for which these classical results hold (see  \cite{orlicz,schur}).

In this note, we aim to unify different versions of Korovkin’s result that incorporated summability methods. Thus, we survey what is  known about this topic, simplyfing  some of the existing proofs and providing a general view on the topic. Of course, Korovkin’s statements are not true for a general summability method; we analyze those summability methods for which Korovkin’s statements continue being true.  Specifically we will show up two properties on a summability method that provides a Korovkin-type approximation result, namely, when the summability method {\it preserves inequalities} or when {\it preserves order-inequalities}. 
Both concept fits a glove to different summability methods and it allow us to obtain several applications.
These  properties  are reminiscent of the squeeze theorem for sequences, as far we know they have been not studied for general summability methods and they deserve subsequent studies.

The paper is structured as follows. In Section \ref{main} we 
 introduce the notation and  we provide a general Korovkin-type result which is true for summability methods that preserve inequalities and order-inequalities. Next, we exhibit a summability method which no preserves inequalities and for which Korovkin's statement fails, that is, these conditions are sharp in some sense. 
 In Section \ref{aplicaciones} we show  several summability methods which are suitable for a Korovkin-type approximation result. Some of them were studied previously: for instance there are versions of Korovkin result, for the statistical convergence \cite{orhan} and the almost summability \cite{almost}. We show some new versions; for instance we provide Korovkin-type approximation results for the convergence method induced by a non-trivial ideal or by any regular matrix summability method with non-negative entries. In brief, we provide a systematic point of view on this topic. 
 The paper concludes with a brief section on concluding remarks and open questions.

\section{Main results}
\label{main}

Let  ${\mathcal R}$ be a summability method on  a Banach space $(X,\|\cdot\|)$. That is, the map $\mathcal{R}\,:\,\mathcal{D}_{\mathcal{R}}\subset  X^{\mathbb{N}}\to X$ assigns to a sequence $(x_n)_{n\geq 1}\in \mathcal{D}_{\mathcal{R}}$ a unique vector $L\in X$. 
To avoid bizarre situations, we will require on $\mathcal{R}$ that the limit assignment does not depend on the first terms, that is, for any  $(x_n)_{n\geq 1}\in \mathcal{D}_{\mathcal{R}}$ such that $\mathcal{R}((x_n)_{n\geq 1})=L$ then for any $n_0\in \mathbb{N}$, we have that $(x_n)_{n\geq n_0}\in \mathcal{D}_{\mathcal{R}}$ and 
$\mathcal{R}((x_n)_{n\geq n_0})=L$.

The following version of the squeeze (or sandwich) theorem will be useful for our objectives.

\begin{definition}
Let $\mathcal{R}\,:\, \mathcal{D}_{\mathcal{R}}\subset X^{\mathbb{N}}\to X$ be a summability method. Let us suppose that  the sequences $(w_n),(x_n),(y_n),(z_n)\in \mathcal{D}_{\mathcal{R}} $ satisfy
\begin{equation}
    \label{des}
\|w_n-w\|\leq C(\|x_n-x\|+\|y_n-y\|+\|z_n-z\|)
\end{equation}
 for some $w,x,y,z\in X$, some constant $C>0$ and for all $n\geq 1$. 
We will say that $\mathcal{R}$ {\sf preserves inequalities} if for any subsequences $(x_n),(y_n), (z_n)$ satisfying (\ref{des}) and satisfying
 $x_n\overset{\mathcal{R}}{\longrightarrow} x$, $y_n\overset{\mathcal{R}}{\longrightarrow} y$, $z_n\overset{\mathcal{R}}{\longrightarrow} z$, we have $w_n\overset{\mathcal{R}}{\longrightarrow} w$.
\end{definition}

The following notion will be also useful.  Let $(E,<)$ be a Banach lattice endowed with a lattice norm  $\|\cdot\|$. As usual, we denote $|x|:=x\vee -x$.
\begin{definition}
Assume that $(E,\leq)$ is a Banach lattice endowed with a lattice norm $\|\cdot\|$. Let $\mathcal{R}\,:\,\mathcal{D}_{\mathcal{R}}\subset X^{\mathbb{N}}\to X$ be a summability method. Assume that there exist sequences $w_n,x_n,y_n,z_n\in \mathcal{D}_{\mathcal{R}}$ satisfying
\begin{align}
\label{desi}
-C[(x_n-x)&+(y_n-y)+(z_n-z)]<w_n-w< \nonumber\\
&< C[(x_n-x)+(y_n-y)+(z_n-z)]
\end{align}
for some constant $C>0$. We will say that $\mathcal{R}$ is {\sf  preserves order-inequalities} if for any sequences $(x_n), (y_n), (z_n)$ satisfying (\ref{desi})  and satisfying  $x_n\overset{\mathcal{R}}{\longrightarrow} x$, $y_n\overset{\mathcal{R}}{\longrightarrow}y$,  $z_n\overset{\mathcal{R}}{\longrightarrow} z$, we have $w_n\overset{\mathcal{R}}{\longrightarrow} w$.
\end{definition}

\begin{theorem}
Let  $\mathcal{R}$ be a summability method on the Banach lattice of continuous function $\mathcal{C}[0,1]$ endowed with the supremum norm. Assume that $\mathcal{R}$ preserves inequalities or order-inequalities. If $(L_n)$ is a sequence of positive linar operators  from $\mathcal{C}([0,1])$ into $\mathcal{C}([0,1])$ then  for any $f\in\mathcal{C}([0,1])$ and bounded on $\mathbb{R}$, $L_nf\overset{\mathcal{R}}{\longrightarrow}f$ if and only if $L_n 1\overset{\mathcal{R}}{\longrightarrow}1$, $L_n t\overset{\mathcal{R}}{\longrightarrow}t$ and $L_n t^2\overset{\mathcal{R}}{\longrightarrow}t^2$.
\end{theorem}
\begin{proof}
Notice that  if  $L_nf\overset{\mathcal{R}}{\longrightarrow} f$ for any $f$, in particular for $e_1=1,e_2=t,e_3=t^2$, we get that $L_ne_i\overset{\mathcal{R}}{\longrightarrow} e_i$, $i=1,2,3$.

For the other implication we proceed as follows. We will show  inequalities similar to (\ref{des}) and (\ref{desi}) and we will use that $\mathcal{R}$ preserve inequalities.
 To this end, we  mimics the original ideas by Korovkin with some details point out by Mohiuddine \cite{almost}. 
Since $f$ is bounded on $\mathbb{R}$ there exists $M>0$ such that
\begin{equation}
\label{uno}
|f(t)-f(x)|< 2M, \quad -\infty<t,x<\infty.
\end{equation}
Since $f$ is continuous uniformly on $[0,1]$, given $\varepsilon>0$ there exists $\delta>0$ such that
\begin{equation}
    \label{dos}
|f(t)-f(x)|< \varepsilon
\end{equation}
whenever $|t-x|<\delta$.  Thus, if we denote $\psi(t)=(t-x)^2$, inequalities (\ref{uno}) and (\ref{dos}) provide
$$
|f(t)-f(x)|< \varepsilon+\frac{2M}{\delta^2}\psi,
$$
which means that
\begin{equation}
    \label{tres}
-\varepsilon-\frac{2M}{\delta^2}\psi< f(t)-f(x)<\varepsilon+\frac{2M}{\delta^2}\psi.
\end{equation}
Now, let us consider $x$ fixed, using that $L_n$ are positive and linear, we get:
\begin{eqnarray}
    \label{cuatro}
-\varepsilon(L_k1)(x) - \frac{2M}{\delta^2}(L_k \psi)(x) & < (L_kf)(x)-f(x)(L_k1)(x) \nonumber\\
&< \varepsilon (L_k 1)(x)+\frac{2M}{\delta^2} (L_k\psi)(x).
\end{eqnarray}
Since
\begin{eqnarray}
\label{cinco}
(L_kf)(x)-f(x)&=(L_kf)(x)-f(x)(L_k1)(x)+f(x)(L_k1)(x)-f(x) \nonumber\\
&=\left[(L_kf)(x)-f(x)(L_k1)(x)\right]+f(x)\left[(L_k1)(x)-1\right],
\end{eqnarray}
using (\ref{cuatro}) and (\ref{cinco}) we get:
\begin{equation}
    \label{seis}
    (L_kf)(x)-f(x)< \varepsilon (L_k1)(x)+\frac{2M}{\delta^2} (L_k\psi)(x)+f(x)\left[(L_k1)(x)-1\right]
\end{equation}
and
\begin{equation}
    \label{inferior1}
 -\varepsilon(L_k1)(x) - \frac{2M}{\delta^2}(L_k \psi)(x)+ f(x)\left[(L_k1)(x)-1\right]  <(L_kf)(x)-f(x).
\end{equation}

Now, let us estimate the term $(L_k\psi)(x)$:
\begin{eqnarray}
    \label{siete}
    (L_k\psi)(x)&=&(L_k (t^2-2tx+x^2))(x)\nonumber\\
    &=&(L_kt^2)(x) - 2x(L_kt)(x) + x^2(L_k1)(x)\nonumber\\
    &=& \left[(L_kt^2)-x^2\right]-2x\left[(L_kt)(x)-x\right]+x^2\left[(L_k1)(x)-1\right].
\end{eqnarray}
Incorporating (\ref{siete}) in (\ref{seis}) we get:

\begin{align*}
    (L_kf)(x)-f(x)&< \varepsilon \left[(L_k1)(x)-1\right]+\varepsilon+\frac{2M}{\delta^2}\left[ \left[(L_kt^2)(x)-x^2\right] \right.\\
   & \left. -2x\left[(L_kt)(x)-x\right]+x^2\left[(L_k1)(x)-1\right] \right]\\
    &+f(x)\left[(L_k1)(x)-1\right].
\end{align*}

 Since $\varepsilon$ is arbitrary and $f$ is bounded, there exists  a constant $C>0$ such that for every $x\in[0,1]$ we get:
$$
(L_kf)(x)-f(x) \leq  C((L_k1)(x)-1)+(L_kt)(x)-x)+(L_kt^2)(x)-x^2).
$$
Analogously, incorporating (\ref{siete}) in (\ref{inferior1}) we obtain:
\begin{align*}
    (L_kf)(x)-f(x)& >-\varepsilon((L_k1)(x)-1)-\varepsilon-\frac{2M}{\delta}\left[((L_kt^2)(x)-x^2)\right.\\
    &\left.-2x((L_kt)(x)-x)+x^2((L_k1)(x)-1) \right]\\
    & +f(x)((L_k1)(x)-1).
\end{align*}
Again, since $\varepsilon$ is arbitrary and $f$ is bounded, there exists a constant $C$ such that
$$
(L_kf)(x)-f(x) \geq  -C((L_k1)(x)-1)+(L_kt)(x)-x)+(L_kt^2)(x)-x^2).
$$
for any $x\in [0,1]$.
Therefore, we obtain the following inequalities
\begin{align*}
-C((L_k1)(x)-1)&+(L_kt)(x)-x)+(L_kt^2)(x)-x^2)<(L_kf)(x)-f(x)\\
& <C((L_k1)(x)-1)+(L_kt)(x)-x)+(L_kt^2)(x)-x^2)
\end{align*}
for any $x\in[0,1]$.
Thus, since by hypothesis $\mathcal{R}$ preserves order-inequalities we get that 
$L_kf\overset{\mathcal{R}}{\longrightarrow}f$ as we desired.

On the other hand, from the above inequalities we deduce that
$$
\|L_kf-f\|_\infty \leq C(\|L_k1-1\|_\infty+\|L_kx-x\|_\infty+\|L_kx^2-x^2\|_{\infty}).
$$ 
Thus, if $\mathcal{R}$ only preserve inequalities we obtain also that  $L_kf\overset{\mathcal{R}}{\longrightarrow}f$ as we desired.
\end{proof}

\begin{remark}
The above proof continues true for positive operators $L_n$ defined from $\mathcal{C}[a,b]$ into $B[a,b]$ the space of real valued bounded functions endowed with the supremum norm. In this situation the summability method $\mathcal{R}$ should be defined on  $B[a,b]$. An elegant proof of Korovkin's result via inequalities by Mitsuru Uchiyama \cite{mitsu} is recommended to the reader.
\end{remark}

\begin{remark}
Let us consider the classical Bernstein polynomials defined by
$$B_nf(x)=\sum_{k=0}^n f\left(\frac{n}{k}\right) \binom{k}{n} x^k (1-x)^{n-k}\quad 0\leq x\leq 1.$$
Denote by $\mathcal{T} \,:\,\mathcal{P}[0,1]\rightarrow \bigvee \{1,t,t^2\}$ the standard  projection and
let us consider the following summability method defined on $\mathcal{P}[0,1]\subset \mathcal{C}[0,1]$ the space of polynomials on $[0,1]$ as follows:
$$
p_n\overset{\mathcal{R}}{\longrightarrow} f
$$
if and only if $\mathcal{T} p_n\to f$ uniformly on $[0,1]$ provided such limit exists.
Now let us observe that $(B_n1)(x)=1, (B_nt)(x)=x$ and $(B_nt^2)(x)=x^2+\frac{x-x^2}{n}$ which implies that
$B_n1\overset{\mathcal{R}}{\longrightarrow} 1$, $B_nt\overset{\mathcal{R}}{\longrightarrow} t$ and $B_nt^2\overset{\mathcal{R}}{\longrightarrow} t^2$.
However, by construction we cannot assert that $B_np\overset{\mathcal{R}}{\longrightarrow} p$ for any polynomial $p$.
The argument in the proof breaks down if we can't guarantee that $\mathcal{R}$ preserves inequalities. This fact highlight the importance of this property in the proof of the above result. 
\end{remark}

The second Korovkin approximation theorem deals with approximation of the identity on the space of $2\pi$-periodic and continuous function on $\mathbb{R}$. We can extract the following result from the original proof by Korovkin. Let us denote by $\mathcal{C}_{2\pi}(\mathbb{R})$ the Banach lattice  of $2\pi$ periodic and continuous functions in $\mathbb{R}$ endowed with the pointwise order and the uniform norm.
\begin{lemma}[Second Korovkin theorem]
Let $L_n$ be a sequence of positive linear operators from $\mathcal{C}_{2\pi}(\mathbb{R})$ into
$\mathcal{C}_{2\pi}(\mathbb{R})$. There exists a constant $C>0$ such that
\begin{align*}
\left|L_n(f)(x)-f(x)\right|&<C\left|L_n(1)(x)-x+L_n(cos(t))(x)-cos(x)\right.\\
&\left.+L_n(sin(t))(x)-sin(x)\right|
\end{align*}
and
\begin{align*}
\|L_n(f)-f\|_{2\pi}&\leq C (\|L_n(1)-1\|_{2\pi}+\|L_n(cos(t))-cos(\cdot)\|_{2\pi}\\
&+\|L_n(sin(t))-sin(\cdot)\|_{2\pi}).
\end{align*}
\end{lemma}
We can deduce the following result:
\begin{theorem}
Let  $\mathcal{R}$  be a summability method on the Banach space  $\mathcal{C}_{2\pi}(\mathbb{R})$ endowed with the supremum norm. Assume that $\mathcal{R}$ preserves inequalities or $\mathcal{R}$ preserves order-inequalities. If $(L_n)$ is a sequence of positive operators  from $\mathcal{C}_{2\pi}(\mathbb{R})$ into $\mathcal{C}_{2\pi}(\mathbb{R})$ then  the following conditions are equivalents:
\begin{enumerate}
    \item For any $f\in \mathcal{C}_{2\pi}(\mathbb{R})$ the sequence $(L_nf)_n$ $\mathcal{R}$-converges to $f$.
    \item $L_n 1\overset{\mathcal{R}}{\longrightarrow}1$, $L_n \sin(t)\overset{\mathcal{R}}{\longrightarrow}\sin(t)$ and $L_n\cos(t) \overset{\mathcal{R}}{\longrightarrow}\cos(t)$.
\end{enumerate}
 
\end{theorem}

\section{Applications}

\label{aplicaciones}

We will see how the results in Section \ref{main} allow us to obtain a unified treatment, visualizing and improving the results on this topic. Let us observe that when $\mathcal{R}$ is the norm convergence then it corresponds with the classical Korovkin's result \cite{altomare}.

{\bf Korovkin's Statement I.} {\it Let $\mathcal{R}\,:\,\mathcal{D}_{\mathcal{R}}\subset\mathcal{C}[0,1]
\to \mathcal{C}[0,1]$ be a summability method. Assume $L_n\,:\,\mathcal{C}[0,1]\to \mathcal{C}[0,1]$
is a sequence of positive linear operators. 
The following conditions are equivalents:
\begin{description}
\item[(i)] For any $f\in\mathcal{C}[0,1]$ and bounded in $\mathbb{R}$, $L_nf\overset{\mathcal{R}}{\longrightarrow}f$.
\item[(ii)] $L_n1\overset{\mathcal{R}}{\longrightarrow}1$, $L_nt\overset{\mathcal{R}}{\longrightarrow}t$ and $L_nt^2\overset{\mathcal{R}}{\longrightarrow}t^2$.
\end{description}}

{\bf Korovkin's Statement II.} {\it Let $\mathcal{R}\,:\,\mathcal{D}_{\mathcal{R}}\subset\mathcal{C}_{2\pi}(\mathbb{R})
\to \mathcal{C}_{2\pi}(\mathbb{R})$ be a summability method. Assume $L_n\,:\,\mathcal{C}_{2\pi}(\mathbb{R})\to \mathcal{C}_{2\pi}(\mathbb{R})$
is a sequence of positive linear operators. 
The following conditions are equivalents:
\begin{description}
\item[(i)] For any $f\in\mathcal{C}_{2\pi}(\mathbb{R})$,  $L_nf\overset{\mathcal{R}}{\longrightarrow}f$.
\item[(ii)] $L_n1\overset{\mathcal{R}}{\longrightarrow}1$, $L_n\sin(t)\overset{\mathcal{R}}{\longrightarrow}\sin(t)$ and $L_n\cos(t)\overset{\mathcal{R}}{\longrightarrow}\cos(t)$.
\end{description}}

In order to establish the above statements we only need to show that the summability method $\mathcal{R}$ preserves inequalities or it preserves order-inequalities.
Some  summability methods are defined through ideals.  Next, we will show that the ideal convergence preserves inequalities. 

{\bf Convergence through ideals. } Let $X$ be a Banach space endowed with the norm $\|\cdot\|$
Let us denote by $\mathcal{P}(\mathbb{N})$ the power set of $\mathbb{N}$.  Let us consider $\mathcal{I}\subset \mathcal{P}(\mathbb{N})$ an arbitrary family of subsets of $\mathbb{N}$. We will say that $\mathcal{I}$ is a non-trivial ideal if 

\begin{enumerate}
    \item $\mathcal{I}\neq \emptyset$ and $\mathcal{I}\neq \mathcal{P}(\mathbb{N})$ (non-triviality).
    \item If $A,B\in\mathcal{I}$ then $A\cup B\in\mathcal{I}$.
    \item If $A\subset B$ and $B\in \mathcal{I}$ then $A\in \mathcal{I}$. 
\end{enumerate}

We say that $\mathcal{I}$ is  regular (or admissible) if it  contains all finite subsets.
A sequence $(x_n)\subset X$ is said to be $\mathcal{I}$-convergent to $L$ (in short we  denote $L=\mathcal{I}-\lim_{n\to\infty}x_n$ or $x_n\overset{\mathcal{I}}{\longrightarrow} L$) if for any $\varepsilon>0$ the subset 
$$A(\varepsilon)=\{n\in \mathbb{N}\,: \, \|x_n-L\|>\varepsilon\}\in \mathcal{I}.$$

Let us see that if  $\mathcal{I}$ denotes the set of all finite subsets  then we it appears the usual convergence. Of course, if $\mathcal{I}\subset \mathcal{J}$ then  $\mathcal{I}$-convergence implies $\mathcal{J}$-convergence. Hence, usual convergence implies $\mathcal{I}$-convergence when  $\mathcal{I}$ is a regular ideal. 
Uniqueness of the limit is not true for $\mathcal{I}$-convergence,
however when $\mathcal{I}$ is non-trivial, $\mathcal{I}$  defines a summability method.

\begin{theorem}
\label{ideal}
Let $\mathcal{I}\subset \mathcal{P}(\mathbb{N})$ be a non-trivial ideal and let us consider the $\mathcal{I}$-convergence on $\mathcal{C}[0,1]$ and on $\mathcal{C}_{2\pi}(\mathbb{R})$. Then, the $\mathcal{I}$-convergence preserves inequalities.
\end{theorem}
\begin{proof}
It is sufficient to show the result for the $\mathcal{I}$-convergence on $\mathcal{C}[0,1]$, the same argument is true on $\mathcal{C}_{2\pi}(\mathbb{R})$.
Indeed, assume that there exists $C>0$ such that
\begin{equation}
    \label{ecuacion}
\|w_n-w\|_\infty\leq C(\|x_n-x\|_\infty+\|y_n-y\|_\infty+\|z_n-z\|_\infty)
\end{equation}
for all $n$,  $x_n\overset{\mathcal{I}}{\longrightarrow}x$, $y_n\overset{\mathcal{I}}{\longrightarrow}y$ and $z_n\overset{\mathcal{I}}{\longrightarrow}z$. We wish to show that $w_n\overset{\mathcal{I}}{\longrightarrow}w$.

Fix $\varepsilon>0$, we wish to show that
$$
A(\varepsilon)=\{n\in\mathbb{N}\,:\, \|w_n-w\|_{\infty}>\varepsilon\}\in\mathcal{I}.
$$
By hypothesis;
$A_1=\{n\in\mathbb{N}\,:\, \|x_n-x\|>\varepsilon/C\}\in \mathcal{I}$, $A_2\{n\in\mathbb{N}\,:\, \|y_n-y\|>\varepsilon/C\}\in \mathcal{I}$ and
$A_3=\{n\in\mathbb{N}\,:\, \|z_n-y\|>\varepsilon/C\}\in \mathcal{I}$.

According to (\ref{ecuacion}) we have that $A(\varepsilon)\subset A_1\cup A_2\cup A_3$, then by applying (2) and (3) in the definition of ideal, we  get that $A(\varepsilon)\in\mathcal{I}$ as  we desired.\end{proof}

 \begin{corollary}
 Korovkin's statements continue being true if we incorporate the  convergence by a non-trivial ideal.
 \end{corollary}

The following trick is folklore and  it allows us to show that the above Korovkin-type approximation result improves the classical one.

\begin{example}
Assume that $\mathcal{I}\subset \mathcal{P}(\mathbb{N})$ is an ideal which contains a non-finite subset, that is, the $\mathcal{I}$-convergence on $\mathcal{C}[0,1]$ is not the uniform convergence.  Let $A\in \mathcal{I}$ non-finite.
Then, let us consider the following sequence
$$
z_n=
\begin{cases}
1 & n\in A\\
0 & n\notin A.
\end{cases}
$$
If we denote by $B_n(f,x)$ the classical Bernstein polynomials, a simple check show that $(1+z_n) B_n(e_i,x)\overset{\mathcal{I}}{\longrightarrow}e_i$ for $e_0=1,e_1=x$ and $e_2=x^2$. However, since $A$ is infinite $(1+z_n)B_n(e_i)$ does not converge to $e_i$ uniformly. That is, the sequence $L_nf=(1+z_n)B_n(f)$ does not satisfy the classical Korovkin theorem.
\end{example}

\begin{remark}
\label{nota}
Let us observe that the above result is also true for a {\sf natural summability method} \cite{unificacion}. Let us recall that
a summability method
$\rho\,:\,D_\rho\subset X^{\mathbb{N}}\to X$  is  natural provided there exists a non-trivial regular ideal $\mathcal{I}\subset \mathcal{P}(\mathbb{N})$ such that in the realm of all bounded sequences  both convergence method $\rho$ and $\mathcal{I}$ are equivalents.   
\end{remark}

 {\bf Statistical Convergence and Strong w$^p$-Ces\`aro convergence. } 
 This is an example when it appears the situation of Remark \ref{nota}. The strong ${\rm w}^p$-Ces\`aro convergence ($0<p<\infty$), is a natural summability method. A sequence $(x_k)$ in a Banach space $X$ is said to be w$^p$-Ces\`aro convergent to $L$ ($p>0$)
if $\lim_{n\to \infty}\frac{1}{n}\sum_{l=1}^n \|x_l-L\|^p=0$. 
Let us denote by $d(\cdot)$ the usual density defined on the subsets of natural numbers. It is well known that the statistical convergence is a kind of ideal convergence, where the ideal is defined by $\mathcal{I}_d=\{ A\subset \mathbb{N}, \,:\, d(A)=0\}$.
Connor established   that  in the realm of all bounded sequences ${\rm w}^p$-convergence and statistical convergence  are equivalent. 
 
 As a consequence of the above considerations we can deduce the following result by 
 Gadjiev and Orhan \cite{orhan}.
 \begin{corollary}
 Korovkin's statements continue being true if we incorporate the Statistical convergence and the Strong w$^p$-Ces\`aro convergence $p>0$.
 \end{corollary}

  {\bf $A$-statistical convergence and $A$-strong  convergence.} Let us consider a matrix $A=(\alpha_{ij})_{(i,j)\in \mathbb{N}\times\mathbb{N}}$ with non-negative entries.  A sequence $(x_i)$
in the space $X$ is said to be $A$-strong summable to $L\in X$, if $$\lim_n \sum_j \alpha_{nj} \|x_j-L\|=.0$$  And the sequence $(x_j)$  is said to be $A$-statistically convergent to $L$ if for any $\varepsilon>0$
$$
\lim_{n\to \infty} \sum_{\substack{j \\ \|x_j-L\|\geq\varepsilon}}
\alpha_{nj}=0.
$$
Both concepts were also introduced by Connor in \cite{connor88} and  he proved  that in the realm of the bounded sequences both notions are equivalents. 
\begin{corollary}
Let $A=(\alpha_{ij})_{(i,j)\in \mathbb{N}\times\mathbb{N}}$ be a matrix with with non-negative entries. Korovkin's statements remain true if we incorporate the $A$-Statistical convergence and the $A$-Strong Ces\`aro convergence.
\end{corollary}
\begin{proof}
First of all, we consider the space  $\mathcal{C}[0,1]$, the proof for the space $\mathcal{C}_{2\pi}(\mathbb{R})$ runs similar.
It is sufficient to show that the $A$-strong Cesàro convergence preserves inequalities.
Indeed, assume that
$$
\|w_n-w\|_\infty\leq C(\|x_n-x\|_\infty+\|y_n-y\|_\infty+\|z_n-z\|_\infty)
$$
then
\begin{align*}
\sum_{j} \alpha_{nj}\|w_j-w\|_\infty&\leq C(\sum_j \alpha_{nj} \|x_j-x\|_\infty+\sum_j \alpha_{nj} \|y_j-y\|_\infty\\
&+\sum_j \alpha_{nj} \|z_j-z\|_\infty),
\end{align*}
taking limits on $n$ we get the desired result.
\end{proof}

{\bf $f$-statistical convergence and $f$-strong Ces\`aro convergence.} Let us recall that $f\,:\, \mathbb{R}^+\to \mathbb{R}^+$ is said to be a modulus function if it satisfies:
\begin{enumerate}
    \item $f(x)=0 $ if and only if $x=0$.
    \item $f(x+y)\leq f(x)+f(y)$ for every $x,y\in \mathbb{R}^+$.
    \item $f$ is increasing.
    \item $f$ is continuous from the right at $0$.
\end{enumerate}
A sequence $(x_n)$ is said to be $f$-statistically convergent to $L$ if for every $\varepsilon>0.$ 
$$
\lim_{n\to \infty}\frac{f(\#\{ k\leq n\,\,:\,\|x_k-L\|>\varepsilon\})}{f(n)}=0
$$
and a sequence $(x_n)$ is said to be $f$-strong Ces\`aro convergent to $L$ if
$$
\lim_{n\to \infty}\frac{f\left(\sum_{k=1}^n\|x_k-L\| \right)}{f(n)}=0.
$$
To avoid trivialities $f$ must be unbounded. While the $f$-statistical convergence is defined by means of a non-trivial ideal, however in general on the realm of bounded sequences both convergence methods are not equivalent (in \cite{jia} are characterized the modulus functions for which this phenomenon happens).

\begin{corollary}
Let $f$ be an unbounded modulus function. Korovkin's statements remain true if we incorporate the $f$-Statistical convergence and the $f$-Strong Ces\`aro convergence.
\end{corollary}
\begin{proof}
Since the $f$-statistical convergence is a special case of ideal-convergence, the result follows by Theorem \ref{ideal}.

Let us prove that the $f$-Strong Cesàro convergence preserves inequalities. It is sufficient to show that on the space $\mathcal{C}[0,1]$. Indeed,  assume that
$$
\|w_n-w\|_\infty\leq C(\|x_n-x\|_\infty+\|y_n-y\|_\infty+\|z_n-z\|_\infty),
$$
without loss we can suppose that the constant $C=k\in \mathbb{N}$ is a natural number, thus by applying the properties (2) and (3) of the modulus $f$, we get:
therefore
\begin{align*}
    f\left(\sum_{l=1}^n \|w_l-w\|_\infty \right) & \leq f\left(\sum_{l=1}^n k\|x_l-x\|_\infty+k\|y_l-y\|_\infty+k\|z_l-z\|_\infty\right)\\
    &\leq kf\left(\sum_{l=1}^n \|x_l-x\|_\infty \right)+kf\left(\sum_{l=1}^n \|y_l-y\|_\infty \right)\\&+kf\left(\sum_{l=1}^n \|z_l-z\|_\infty \right),
\end{align*}
dividing by $f(n)$ and taking limits as $n$ tends to infinite, we prove the desired result.
\end{proof}

{\bf Some summability methods that preserve order-inequalities. }
In this paragraph we analyse some summability methods that preserve order inequalities, thus Korovkin statements continue being true if we incorporate these summability methods.

A sequence $(f_k)$ in a Banach lattice $(X,<)$ endowed with the lattice norm $\|\cdot\|$ is said to be almost convergent to $L\in X$ if the double sequence
$$
\frac{1}{m+1}\sum_{i=0}^mx_{n+i}
$$
converges to $L$ as $m\to\infty$ uniformly in $n$.

\begin{theorem}
Let $(X,<)$ be a Banach lattice with the lattice norm $\|\cdot\|$. Let us consider $\mathcal{R}$ the almost summability defined on $\mathcal{D}_{\mathcal{R}}\subset X^{\mathbb{N}}$. Then $\mathcal{R}$ preserves order-inequalities.
\end{theorem}
\begin{proof}
Indeed, assume that there exist sequences $(w_i),(x_i),(y_i)$ and $(z_i)$ satisfying
\begin{align*}
-C\left((x_i-x)+(y_i-y)+(z_i-z)\right) &<w_i-w \\
&< C\left((x_i-x)+(y_i-y)+(z_i-z)\right).
\end{align*}
Thus,
\begin{align*}
\left|\frac{1}{m+1}\sum_{i=0}^mw_{n+i}-w\right|&<C\left|\frac{1}{m+1}\sum_{i=0}^mx_{n+i}-x+\frac{1}{m+1}\sum_{i=0}^my_{n+i}-y\right. \\&+\left.\frac{1}{m+1}\sum_{i=0}^mz_{n+i}-z\right|,
\end{align*}
hence, taking norms
\begin{align*}
\left\|\frac{1}{m+1}\sum_{i=0}^mw_{n+i}-w\right\|&<C\left[\left\|\frac{1}{m+1}\sum_{i=0}^mx_{n+i}-x\right\|+\left\|\frac{1}{m+1}\sum_{i=0}^my_{n+i}-y\right\|\right. \\
&+\left.\left\|\frac{1}{m+1}\sum_{i=0}^mz_{n+i}-z\right\|\right].
\end{align*}
Letting $m\to \infty$ we get
$$
\left\|\frac{1}{m+1}\sum_{i=0}^mw_{n+i}-w\right\|\to 0
$$
uniformly in $n$, therefore $w_n\overset{\mathcal{R}}{\longrightarrow}w$ as we desired.
\end{proof}
\begin{corollary}
 Korovkin's statements continue being true if we incorporate the almost summability.
\end{corollary}

Let $A=(\alpha_{ij})_{(i,j)\in \mathbb{N}\times\mathbb{N}}$ be a matrix with  non-negative entries.   A sequence $(x_i)$ in a Banach lattice $(X,<)$ is said to be $A$-summable (or $A$-convergent) to $L\in X$, if $\lim_n \sum_j \alpha_{nj} x_j=L$.  
A non-negative matrix $A$ is said to be regular if
\begin{enumerate}
    \item $\sup_n\sum_j\alpha_{nj}<\infty$.
    \item $\lim_n \alpha_{ni}=0$.
    \item $\lim_n\sum_i\alpha_{ni}=1$.
\end{enumerate}

\begin{theorem}
Let $(X,<)$ be a Banach lattice with the lattice norm $\|\cdot\|$. Let $A$ be a regular matrix with non-negative entries. The $A$-convergence   preserves order-inequalities.
\end{theorem}
\begin{proof}
Indeed, assume that there exist sequences $(w_i),(x_i),(y_i)$ and $(z_i)$ satisfying
\begin{align*}
-C\left((x_i-x)+(y_i-y)+(z_i-z)\right) &<w_i-w< \\
&< C\left((x_i-x)+(y_i-y)+(z_i-z)\right).
\end{align*}
Since $e_n\overset{A}{\longrightarrow}e$, for $e_j=x_j,y_j,z_j$, $e=x,y,z$ and $\lim_n\sum_i\alpha_{ni}=1$ we get:
$$
\left\|\sum_j \alpha_{nj} e_j-\sum_{j}a_{nj}e\right\|\leq 
\left\|\sum_j \alpha_{nj} e_j-e\right\|+\left|1-\sum_{j}a_{nj}\right| \|e\|\to 0
$$
as $n\to\infty$ for $e=x,y,z$.

Since the entries of $A$ are non-negative, for every $n$ we get
\begin{align*}
\left|\sum_j \alpha_{nj} w_j-\left(\sum_j \alpha_{nj}\right)w \right|&\leq C\left| \sum_j \alpha_{nj} x_j-\left(\sum_{j}a_{nj}\right)x\right.\\
&+\sum_j \alpha_{nj} y_j-\left(\sum_{j}a_{nj}\right)y\\
&\left.+\sum_j \alpha_{nj} z_j-\left(\sum_j\alpha_{nj}\right)z\right| .
\end{align*}
Taking norms and applying the triangular inequality we get that
$$
\lim_{n\to \infty }\left\|\sum_j \alpha_{nj} w_j-\left(\sum_{j}\alpha_{nj}\right)w\right\|=0.
$$
Hence, since
$$
\left\|\sum_j \alpha_{nj} w_j-w\right\|\leq 
\left|\sum_j \alpha_{nj} -1\right|\|w\|+\left\|\sum_j \alpha_{nj} w_j-\left(\sum_{j}\alpha_{nj}\right)w\right\|,
$$
and each summand in the right hand  converges in norm to zero as $n$ tends to $\infty$, we get that $w_n\overset{A}{\longrightarrow}w$
as we desired.
\end{proof}
\begin{corollary}
 Korovkin's statements remain true if we incorporate a regular matrix summability method with non-negative entries.
\end{corollary}

\section{Concluding remarks and open questions}

Korovkin-type approximation theorems have been successfully obtained for functions of two variables by and for different types of convergence \cite{sev1,sev2,sev3}, including summability methods on double sequences.

We believe that some of the structure discovered here in one variable remain true for functions of  two variables and for  summability method on double sequences. It will be interesting to obtain a result which could unify everything.

\begin{quote}
{\bf Question.} Does it is possible a unification of Korovkin-type approximation results in the several variables setting?
\end{quote}

In the search for new Korovkin type results, it would be interesting to obtain Korovkin's statement for a summability method that does not preserve inequalities or order-inequalities. That is, the following question arises:

\begin{quote}
{\bf Question. } Let $A$ be a matrix defining a non-necessarily regular summability method. Does is it possible a Korovkin-type result incorporating the $A$-summability? 
In general, for which matrices $A$ does is it possible to obtain a Korovkin-type result?
\end{quote}

{\bf Acknowledgements. } The authors were supported by by Ministerio de Ciencia, Innovaci\'on y Universidades under grant PGC2018-101514-B-I00.

\section*{Availability of data and material}
Not applicable.

\section*{Acknowledgements}
We want to thank to the Vicerretorado de Investigación of University of Cádiz who supported  partially this research.

\section*{Author's contribution}
MCLG contributed significantly to analysis and manuscript preparation; PR helped perform the analysis with constructive discussions. All authors read and approved the final manuscript.

\section*{Compliance with Ethical Standards}
The author Maria del Carmen Listán García declares that she has no conflict of interest.
The author Maria del Romero de la Rosa declares that she has no conflict of interest.
\section*{Funding}

 The authors were supported by by Ministerio de Ciencia, Innovaci\'on y Universidades under grant PGC2018-101514-B-I00.


\end{document}